\documentclass[letterpaper, 10pt, conference]{ieeeconf}  
\IEEEoverridecommandlockouts                              
\usepackage{graphics}
\usepackage{epstopdf} 

\title{
\LARGE{ 
\textbf{
Local Convergence Behaviour of Generalized Gauss-Newton Multiple Shooting, Single Shooting and Differential Dynamic Programming}}}

\author{
	Katrin Baumgärtner, Florian Messerer, Moritz Diehl
	\thanks{This research was supported by DFG via Research Unit FOR 2401 and project 424107692 and by the EU via ELO-X 953348.}
	\thanks{Katrin Baumg{\"a}rtner and Florian Messerer are with the Department of Microsystems Engineering (IMTEK) and Moritz Diehl is with the Department of Microsystems Engineering (IMTEK) and Department of Mathematics, University Freiburg, 79110 Freiburg, Germany.}%
	\thanks{katrin.baumgaertner@imtek.uni-freiburg.de}%
}

\usepackage{stackengine}
\usepackage{xcolor}

\usepackage{array,multirow, makecell, graphicx}

\newcommand{\gray}[1]{\textcolor{gray}{#1}}

\usepackage{mathrsfs, amsmath, amssymb, mathtools, amsfonts}
\DeclareMathSymbol{\shortminus}{\mathbin}{AMSa}{"39}
\usepackage{optidef}

\newcommand{\T}{^{\!\top}\!}
\newcommand{\I}{\mathbb{I}}
\newcommand{\inv}{^{\shortminus 1}}
\newcommand{\R}{\mathbb{R}}

\DeclareMathAlphabet{\mymathbb}{U}{bbold}{m}{n}
\newcommand{\zeros}{\mymathbb{0}}

\newcommand{\bA}{\bar{A}}
\newcommand{\bB}{\bar{B}}

\newcommand{\ms}{\textsc{ms}}
\newcommand{\sss}{\textsc{ss}}
\newcommand{\ddp}{\textsc{ddp}}

\newcommand{\xzero}{\bar{\bar{x}}_0}

\newcommand{\dpartial}[2]{\frac{\partial #1}{\partial #2}}

\newcommand{\dtotal}[2]{\frac{\mathrm{d} #1}{\mathrm{d} #2}}

\newcommand{\bw}{\bar{w}}
\newcommand{\by}{\bar{y}}
\newcommand{\bz}{\bar{z}}
\newcommand{\bu}{\bar{u}}
\newcommand{\bx}{\bar{x}}
\newcommand{\blam}{\bar{\lambda}}

\newcommand{\MS}{multiple shooting }
\newcommand{\CMS}{Multiple shooting }
\renewcommand{\SS}{single shooting }
\newcommand{\CSS}{Single shooting }

\newcommand{\wildcard}{\star}

\newcommand{\grad}[1]{\nabla_{\! #1}}

\newcommand{\bigO}{\mathcal{O}}

\newcommand{\sm}{\shortminus}

\newcommand{\ggn}{\textsc{ggn}}
\newcommand{\exact}{\textsc{eh}}

\renewcommand{\nu}{{n_u}}

\newcommand{\Flin}{\hat{F}}

\newcommand{\qua}{\mathrm{quad}}

\usepackage{mathtools} 

\usepackage{multirow}

\usepackage{amsthm}
\newtheorem{theorem}{Theorem}
\newtheorem{lemma}{Lemma}
\newtheorem{remark}{Remark}
\newtheorem{corollary}{Corollary}
\newtheorem{proposition}{Proposition}

\theoremstyle{definition}

\usepackage{algorithm,algorithmic}

\usepackage{tikz}
\usetikzlibrary{shapes,
                backgrounds,
                arrows,
                patterns,
                automata,
                trees,
                calc,
                positioning, 
                matrix}            

\usepackage{booktabs}
\usepackage[per-mode=symbol]{siunitx}

\makeatletter
\newcommand*\bigcdot{\mathpalette\bigcdot@{.7}}
\newcommand*\bigcdot@[2]{\mathbin{\vcenter{\hbox{\scalebox{#2}{$\m@th#1\bullet$}}}}}
\makeatother

\makeatletter
\renewcommand*\env@matrix[1][c]{\hskip -\arraycolsep
  \let\@ifnextchar\new@ifnextchar
  \array{*\c@MaxMatrixCols #1}}
\makeatother

\begin{document}

\maketitle
\thispagestyle{empty}
\pagestyle{empty}

\begin{abstract}
We revisit three classical numerical methods for solving unconstrained optimal control problems
--
multiple shooting, single shooting, and differential dynamic programming
--
and examine their local convergence behaviour.
In particular, we show that all three methods converge with the same linear rate if a Gauss-Newton (GN), or more general a Generalized Gauss-Newton (GGN), Hessian approximation is used, which is the case in widely used implementations such as iLQR.

\end{abstract}

\section{Introduction}

Multiple shooting (MS), single shooting (SS) and differential dynamic programming (DDP) are three numerical methods that might be used for solving discrete optimal control problems (OCP) that typically arise after discretization of a continuous-time OCP and take the form:
\begin{mini!}|s|
	{\scriptstyle{x, u} }
	{ \sum_{i=0}^{N-1} l_i(x_i, u_i) + l_N(x_N) }
	{\label{eq:nlp}}{}
	\addConstraint{x_0}{= \xzero}
	\addConstraint{x_{i+1} }{= f_i(x_i, u_i),}{~i=0, \ldots, N-1,}
\end{mini!}
with states $x = (x_0, \ldots, x_N)$, $x_i \in \R^{n_x}$, controls $u = (u_0, \ldots, u_{N-1})$, $u_i  \in \R^{n_u}$ and a given initial state $\xzero$.

This multiple shooting formulation as given in \eqref{eq:nlp} keeps both the controls and the states as optimization variables.
Within the  Sequential Quadratic Programming (SQP) framework, the corresponding quadratic subproblems can be solved efficiently via the Riccati recursion due to the special structure of the OCP  \cite{Rao1998}.

While \MS is a simultaneous approach -- it solves the simulation and optimization problem simultaneously -- both single shooting and DDP can be considered sequential approaches.
\CSS eliminates the states via forward simulation and keeps only the control inputs as optimization variables yielding an unconstrained nonlinear program (NLP).
If \SS  is implemented in a sparsity-exploiting fashion, quadratic subproblems with the same sparse structure as in \MS need to be solved \cite{Sideris2005}.
Based on the controls obtained from the solution of this subproblem, an additional open-loop simulation of the nonlinear system dynamics needs to be performed.
Similarly, DDP can be implemented by first performing a Riccati recursion based on the very same quadratic subproblem and then simulating the nonlinear system forward in time.
In contrast to the open-loop simulation performed within single shooting, DDP leverages the time-varying affine feedback law that is obtained from the Riccati recursion within the nonlinear forward simulation.

Depending on the choice of Hessian approximation that is chosen for the quadratic subproblems, the three methods come in different algorithmic variants.
Assuming convex stage and terminal costs, we consider two common Hessian approximations:
exact Hessian (EH) and the Generalized Gauss-Newton (GGN) Hessian approximation.
The GGN Hessian is a generalization of the Gauss-Newton (GN) Hessian, which is widely used in case of quadratic stage and terminal costs, to general convex cost functions \cite{Schraudolph2002, Messerer2021a}.
Assuming that the iterates converge, all three methods locally converge with a quadratic rate if the exact Hessian is used.
With a GGN Hessian approximation, the local convergence rate is in general linear and, as we will show in the following, the asymptotic rate of convergence is the same for all three methods.

\subsection{Contribution \& Outline}

The contribution of this paper is to provide a unified view on multiple shooting, \SS and DDP from a numerical optimization perspective.
In particular, we show that the GGN variants of the three methods locally converge at the same linear rate.
This rate can be exactly characterized as the smallest scalar that satisfies two linear matrix inequalities.

After providing an overview on related work in the next paragraph, we briefly recall the three algorithms in Section~\ref{sec:methods} highlighting their similarities and differences. 
Section~\ref{sec:convergence} analyzes the local convergence behaviour of the three methods, which is illustrated on a simple example in Section~\ref{sec:examples}.



\subsection{Related Work}

The DDP algorithm using exact Hessians was originally proposed by Mayne in 1966 \cite{Mayne1966} and further analyzed in \cite{Jacobson1970}.
Proofs for quadratic convergence of DDP were first given in 1984 by \cite{Murray1984} and \cite{Pantoja1984}.
In 1990, Shoemaker and Liao provided a proof based on Bellman's principle of optimality \cite{Shoemaker1990}.

Its Gauss-Newton variant, which is more commonly referred to as \textit{iterative Linear Quadratic Regulator} (iLQR), especially within the robotics community \cite{Tassa2014}, has been introduced in \cite{Li2004a, Todorov2005}.

The sparsity-exploiting implementation of single shooting, which we consider here, has first been introduced in \cite{Sideris2005, Sideris2005a} using a Gauss-Newton Hessian. 
The Gauss-Newton variant has also been analyzed more recently in \cite{Roulet2019}.

In the context of direct optimal control, \MS was first suggested by Bock in 1984  \cite{Bock1984}.
Even earlier, the \MS approach has been discussed for parameter identification  \cite{Bock1983} and boundary value problems \cite{Osborne1969}.

The quadratic convergence behaviour of the exact Hessian variant of both multiple and single shooting directly follows from the analysis of Newton's method.
For the Gauss-Newton variants, local linear convergence has first been analyzed in \cite{Bock1983}.
For the Generalized Gauss-Newton variants, we refer to \cite{Messerer2021a} for a detailed analysis.

For multiple and single shooting, the exact characterization of their local contraction rate follows directly from the results in \cite{Messerer2021a, Messerer2020}, where general Sequential Convex Programming and Generalized Gauss-Newton methods are considered.

In \cite{Giftthaler2018a}, a family of Gauss-Newton shooting methods is introduced that combine the multiple shooting approach on a coarse discretization grid with Gauss-Newton DDP or Gauss-Newton single shooting, which is performed on a fine discretization grid within each multiple shooting interval.
Our analysis can be easily extended to this family of algorithms.

In \cite{Albersmeyer2010}, the local convergence behaviour of \MS and \SS with exact Hessians has been discussed in a simplified setting, which shows different quadratic rates for the two methods.

\section{Unconstrained Optimal Control Problem and Numerical Methods}
\label{sec:methods}

In this section, we briefly recall the three numerical methods and point out their similarities and differences.

We consider optimal control problems of the form given in \eqref{eq:nlp},
where we assume that the stage costs $l_i$, as well as the terminal cost $l_N$ are convex.
If we linearize the dynamics and approximate the objective by a quadratic function at the current iterate $(\bx, \bu)$ – or $(\bx, \bu, \blam)$, where $\blam$ are the dual variables, for the exact Hessian variant –, we obtain an equality constrained Quadratic Program (QP),
\begin{mini!}|s|
	{ \scriptstyle{x, u}}
	{ 
		\sum_{i=0}^{N-1} 
		\begin{bmatrix}
			q_i^\wildcard  \\ r_i^\wildcard 
		\end{bmatrix}\T \!
		\begin{bmatrix}
			x_i \\ u_i
		\end{bmatrix} 
		+ \frac{1}{2} 
		\begin{bmatrix}
			x_i \\ u_i
		\end{bmatrix}\T \!
		\begin{bmatrix}
			Q_i^\wildcard & (S_i^\wildcard)\T \\ 
			S_i^\wildcard & R_i^\wildcard
		\end{bmatrix} \!
		\begin{bmatrix}
			x_i \\ u_i
		\end{bmatrix} \nonumber
	}
	{\label{eq:qp}}{}
	\breakObjective{
		\qquad +p_N\T\, x_{N} + \frac{1}{2} x_{N}\T P_N x_{N} }
	\addConstraint{x_0}{= \xzero}
	\addConstraint{x_{i+1} }{= a_i + A_i x_i + B_i u_i ,~i=0, \ldots, N-1,}
\end{mini!}
where the linearized dynamics are given by
\begin{align} 
A_i &= \dpartial{f_i}{x_i}(\bx_i, \bu_i), \quad 
B_i = \dpartial{f_i}{u_i}(\bx_i, \bu_i), 
\label{eq:linA} \\
a_i &= f_i(\bx_i, \bu_i) - A_i \bx_i - B_i \bu_i.
\label{eq:lina}
\end{align}
The cost gradients are
\begin{align}
q_i^\wildcard  &= \grad{x_i} l(\bx_i, \bu_i) - Q_i^\wildcard  \bx_i - (S_i^{\wildcard})^{\top} \bu_i,  
\label{eq:linq}\\
r_i^\wildcard  &= \grad{u_i} l(\bx_i, \bu_i) -  S_i^\wildcard  \bx_i - R_i^{\wildcard} \bu_i, 
\label{eq:linr} \\
p_N &= \grad{x_N} l_N(\bx_N) - P_N \bx_N.
\label{eq:linpN}
\end{align}

The Hessian associated with the terminal stage is given by
\begin{align}
P_N &= \grad{x_i}^2 l_N(\bx_N).
\label{eq:linPN}
\end{align}
The Hessian blocks for the all other stages are given by
\begin{align}\label{eq:linQRSggn}
\begin{bmatrix}
Q_i^\ggn & (S_i^\ggn)\T \\ 
S_i^\ggn & R_i^\ggn
\end{bmatrix}
= \grad{(x_i, u_i)}^2 ~ l_i(\bx_i, \bu_i)
\end{align}
if a Generalized Gauss-Newton Hessian approximation is used, and by
\begin{align} \label{eq:linQRSeh}
\begin{bmatrix}
Q_i^\exact & (S_i^\exact)\T \\ 
S_i^\exact & R_i^\exact
\end{bmatrix}
\!\!&=\! \grad{(x_i, u_i)}^2 \left(l_i(\bx_i, \bu_i) \!+\! \blam_{i+1}\T f_i(\bx_i, \bu_i) \right)
\end{align}
if the exact Hessian is used.
Note that the dual variables $\bar \lambda_i$, associated with the equality constraints in \eqref{eq:nlp}, are required if an exact Hessian is used, while they need not be computed for the GN and GGN variant. 


\begin{table*}
\vspace*{10pt}
\caption{Backward and forward sweep of DDP, \MS and \SS in comparison.}
\label{tab:overview}
\vspace*{-10pt}
\hrulefill
\begin{subequations}
\fontsize{7.7}{9}\selectfont
\begin{alignat}{6}
\multispan2{\textsc{Input:} $\bar{x}, \bar{u}$ (feasible) \hfil} & 
\multispan2{\textsc{Input:} $\bar{x}, \bar{u}, \gray{\bar{\lambda}}$\hfil} & \multispan2{\textsc{Input:} $\bar{x}, \bar{u}$ (feasible)\hfil} 
\nonumber 
\\[8pt]
\multispan2{\textsc{DDP – backward sweep} \hfil} & 
\multispan2{\textsc{Multiple Shooting – backward sweep} \hfil} & \multispan2{\textsc{Single Shooting – backward sweep}\hfil} 
\nonumber 
\\[8pt]
\bar{P}_{N}, \bar{p}_N & = \text{ via eq. \eqref{eq:linPN} and \eqref{eq:linpN}}  &
\bar{P}_{N}, \bar{p}_N & = \text{ via eq. \eqref{eq:linPN} and \eqref{eq:linpN}}  &
\bar{P}_{N}, \bar{p}_N & = \text{ via eq. \eqref{eq:linPN} and \eqref{eq:linpN}}  &
\\
\gray{\bar{\lambda}_{N}} & \gray{= p_N + P_N \bar{x}_N,} &
&&
\gray{\bar{\lambda}_{N}} & \gray{= p_N + P_N \bar{x}_N,} 
\\[5pt]
A_i\!, B_i\!, a_i\! & = \text{ via eq. \eqref{eq:linA} and \eqref{eq:lina}} &
A_i\!, B_i\!, a_i\! & = \text{ via eq. \eqref{eq:linA} and \eqref{eq:lina}} &
A_i\!, B_i\!, a_i\! & = \text{ via eq. \eqref{eq:linA} and \eqref{eq:lina}}
\\
Q_i\!, R_i\!, S_i\! & = \text{ via eq. \eqref{eq:linQRSggn} \gray{or \eqref{eq:linQRSeh}}} &
Q_i\!, R_i\!, S_i\! & = \text{ via eq. \eqref{eq:linQRSggn} \gray{or \eqref{eq:linQRSeh}}} &
Q_i\!, R_i\!, S_i\! & = \text{ via eq. \eqref{eq:linQRSggn} \gray{or \eqref{eq:linQRSeh}}} 
\\
q_i, r_i & = \text{ via eq. \eqref{eq:linq} and \eqref{eq:linr}}  &
q_i, r_i  & = \text{ via eq. \eqref{eq:linq} and \eqref{eq:linr}}&
q_i, r_i  & = \text{ via eq. \eqref{eq:linq} and \eqref{eq:linr}}
\\
P_i, p_i & = \text{ via eq. \eqref{eq:P-recursion} and \eqref{eq:p-recursion}} &
P_i, p_i & = \text{ via eq. \eqref{eq:P-recursion} and \eqref{eq:p-recursion}} &
P_i, p_i & = \text{ via eq. \eqref{eq:P-recursion} and \eqref{eq:p-recursion}} 
\\
K_i, k_i & = \text{ via eq. \eqref{eq:K} and \eqref{eq:k}} &
K_i, k_i & = \text{ via eq. \eqref{eq:K} and \eqref{eq:k}} &
K_i, k_i & = \text{ via eq. \eqref{eq:K} and \eqref{eq:k}}
\\
\gray{\bar{\lambda}_{i}} & \gray{= p_i + P_i \bar{x}_i,} &
&&
\gray{\bar{\lambda}_{i}} & \gray{= p_i + P_i \bar{x}_i,}
\\
\multispan2{where $i=N-1, \ldots, 0$,\hfil} &
\multispan2{where $i=N-1, \ldots, 0$,\hfil} &
\multispan2{where $i=N-1, \ldots, 0$,\hfil} \nonumber \\[12pt]
\multispan2{\textsc{DDP – forward sweep} \hfil} & \multispan2{\textsc{Multiple Shooting – forward sweep}\hfil} & \multispan2{\textsc{Single Shooting – forward sweep}\hfil} 
\nonumber \\[8pt]
x_0 & = \xzero, &
x_0 & = \xzero, &
x_0 & = \xzero, 
\\
u_i & = \bu_i + k_i + K_i(x_i - \bx_i), \quad\qquad &
u_i & = \bu_i + k_i + K_i(x_i - \bx_i), &
u_i & = \bu_i + k_i + K_i(\hat{x}_i - \bx_i), 
\label{eq:u-forward}\\
&&
x_{i+1} & = \bar{f}_i + A_i (x_i - \bx_i) + B_i (u_i - \bu_i), \quad\qquad&
\hat{x}_{i+1} & = \bar{f}_i + A_i (\hat{x}_i - \bx_i) + B_i (u_i - \bu_i), \qquad
\\
x_{i+1} & = f(x_i, u_i), &
&&
x_{i+1} & = f(x_i, u_i),
\\[3pt]
\multispan2{where $i=0, \ldots, N-1.$\hfil} 
&
\multispan2{where $i=0, \ldots, N-1.$\hfil} 
&
\multispan2{where $i=0, \ldots, N-1.$\hfil} 
\nonumber \\[3pt]
&&
\gray{\lambda_{i}} & \gray{= \bar p_i + \bar P_i x_i,} &
\\[3pt]
&&
\multispan2{where $i=0, \ldots, N.$\hfil} &
\nonumber \\[8pt]
\multispan2{\textsc{Output:} $x, u$ \hfil} & 
\multispan2{\textsc{Output:} $x, u, \gray{\lambda}$\hfil} & \multispan2{\textsc{Output:} $x, u$ \hfil} 
\nonumber 
\end{alignat}
\label{eq:summary}
\end{subequations}
\hrulefill
\end{table*}

\CMS solves an instance of the quadratic subproblem given in \eqref{eq:qp} in every iteration.
A summary of \MS algorithm is given in Table~\ref{tab:overview}, in the center column.
Due to the particular structure of the QP, it can be efficiently solved via a backward Riccati recursion and a forward simulation based on the linearized dynamics.
In particular, the method proceeds as follows: 
First, the nonlinear OCP is linearized at the current iterate, \eqref{eq:linA} to \eqref{eq:linQRSeh}, using an exact or GGN Hessian.
Next, a Riccati recursion is performed, which proceeds as follows:
\begin{flalign}
K_i \!& = - (R_i \!+\! B_i\T P_{i+1} B_i )\inv 
(S_i \!+\! B_i\T P_{i+1} A_i),
\label{eq:K}\\
 k_i \!&= - ( R_i \!+\! B_i\T P_{i+1} B_i )\inv
(r_i \! +\! B_i\T (P_{i+1} a_i\! +\! p_{i+1})), &
\label{eq:k} 
\\[8pt]
 P_i \!& =  Q_i + A_i\T  P_{i+1} A_i 
+
( S_i\T + A_i\T  P_{i+1} B_i)
K_i,
\label{eq:P-recursion}\\
 p_i \!& = 
q_i + A_i\T( P_{i+1} a_i + p_{i+1}) 
\nonumber \\
& \quad ~~~\,
+ K_i\T
( r_i + B_i\T( P_{i+1} a_i + p_{i+1})),
\label{eq:p-recursion}
\end{flalign}
for $i = N-1, \ldots, 0$.
Finally, a forward simulation is performed using the linearized system dynamics and the linear feedback law defined by $K_i, k_i$.
If multipliers are required, they are updated as well at this final step. 

Now turning to single shooting and DDP, summarized in the left and right column of Table~\ref{tab:overview}, we first point out that both DDP and single shooting require a feasible initial guess, which is not the case for multiple shooting.
Furthermore, note that for multiple shooting, the three steps – linearization, backward Riccati recursion, forward sweep – could be implemented sequentially.
If the exact Hessian is used, this does not hold for single shooting and DDP, where the Hessian blocks $Q_i$, $R_i$, $S_i$ of stage $i$ depend on the multiplier $\bar\lambda_{i+1}$ computed in the previous recursive step of the very same backward sweep. 
With multiple shooting, the multipliers are part of the \textit{memory} of the algorithm, which is not the case for \SS and DDP, where they need to be computed \textit{on the fly}.
Thus, linearization and backward Riccati recursion are entwined and cannot be implemented sequentially.

After the backward sweep, single shooting performs a linear forward sweep to obtain the controls and a nonlinear \textit{open-loop} simulation to obtain the states.
In contrast, DDP uses the nonlinear system dynamics as well as the linear feedback law defined by $K_i, k_i$ to perform a \textit{closed-loop} forward simulation of the nonlinear system.  



\section{Convergence Analysis}
\label{sec:convergence}

In this section, we analyze the local convergence behaviour of multiple shooting, single shooting and DDP. 
In particular, we show that all three methods have the same linear contraction rate if a GGN Hessian approximation is used.
In fact, our analysis extends to any Hessian approximation based on the primal variables $(x_i, u_i)$.

We define the primal-dual iterate $z=(x, u, \lambda)$ where $\lambda$ are the multipliers associated with the equality constraints.

\begin{proposition}
Let $z^* = (x^*, u^*, \lambda^*)$ be a feasible point of \eqref{eq:nlp} at which LICQ holds.
The following statements are equivalent:
\begin{enumerate}[]
	\item[(i)] $z^*$ is KKT point of the NLP in \eqref{eq:nlp}.
	\item[(ii)] $z^*$ is a fixed point of the \MS iteration.
	\item[(iii)] $z^*$ is a fixed point of the \SS  iteration.
	\item[(iv)] $z^*$ is a fixed point of the DDP iteration.
\end{enumerate}
\end{proposition}
\begin{proof}
We refer to, e.g., Chapter 8.8 in \cite{Rawlings2017}.
\end{proof}

All three algorithms can be defined in terms of a nonlinear parametric root-finding problem, which we will do in the following.
In a neighbourhood of a solution, the convergence behaviour of the iterates is governed by the spectral radius of the Jacobian of the solution map, which is shown in the following classical result:

\begin{theorem} \label{thm:local-contraction}
Let $\Pi:\R^{n_z} \rightarrow \R^{n_z}$ be the solution map of the nonlinear parametric root-finding problem $R(z; \bar{z}) = 0$ such that $R(\Pi(\bar{z}); \bar{z}) = 0$ and with $R$ twice continuously differentiable.

Suppose that $z^*$ is a fixed point of the iteration $z^+ = \Pi(z)$ with $\dpartial{R}{z_1}(z^*; z^*)$ nonsingular.
Let $\kappa(z^*) := \rho(J(z^*))$ be the spectral radius
of the matrix $J(z^*)$ given as
\begin{align*}
J(z^*) := -\left(\dpartial{R}{z}(z^*; z^*)\right) \inv \dpartial{R}{\bar{z}}(z^*; z^*).
\end{align*}
If $0 < \kappa(z^*) < 1$, the iterates locally converge to $z^*$ at a linear rate. 
The asymptotic convergence rate is given by $\kappa(z^*)$. 
If $\kappa(z^*) = 0$, the iterates locally converge to $z^*$ at a superlinear rate.
If $\kappa(z^*) > 1$, the fixed point $z^*$ is unstable. 
\end{theorem}

\begin{proof} 
A Taylor expansion of the solution map  $\Pi(z)$ at the fixed point $z^*$ yields
\begin{align*}
z^{k+1}-z^* 
& = \dtotal{\Pi}{\bar{z}} (z^*) (z^k - z^*)  
+ \mathcal{O}\left(\Vert z^k - z^* \Vert^2\right).
\end{align*}
The derivative $\dtotal{\Pi}{\bar{z}} (z^*)$ is given by the matrix $J(z^*)$, which follows from the implicit function theorem.
A standard result of linear stability analysis of nonlinear systems shows that local convergence of the iterates $z^k$ to $z^*$ is determined by the spectral radius $\rho(J(z^*))$, (compare e.g. \cite{Ostrowski1973}).


\end{proof}

The following lemma will allow us to show that the matrix $J(z^*)$ in Theorem~\ref{thm:local-contraction} is the same for all three methods.

\begin{lemma} \label{lem:partial-derivatives}
We consider two twice continuously differentiable functions $R_1: \R^m \times \R^m \rightarrow \R^l$, $(x, y) \mapsto R_1(x,y)$ and $R_2:\R^m \times \R^m \rightarrow \R^l$, $(x, y) \mapsto R_2(x,y)$.
If it holds that 
\begin{align*}
R_1(z, z) = R_2(z,z), \qquad
\dpartial{R_1}{x} (z, z) = \dpartial{R_2}{x}(z, z),
\end{align*}
for all $z \in \R^m$, then 
$R_1(x, y) - R_2(x, y) = \bigO\left(\Vert x - y\Vert^2\right)$,
and in particular
$\dpartial{R_1}{y} (z, z) = \dpartial{R_2}{y}(z, z)$.
\end{lemma}

\begin{proof}
A first-order Taylor expansion of $R_1(x, y)$ in $x$ around the linearization point $\bx \in \R^m$ yields
\begin{align*} 
\begin{aligned}
R_1(x, y) &= \!R_1(\bx, y) + \tfrac{\partial R_1}{ \partial x}(\bx, y)(x - \bx) + \bigO\!\left(\Vert x - \bx\Vert^2\right)\!,
\\
R_2(x, y) &= \!R_2(\bx, y) + \tfrac{\partial R_2}{\partial x}(\bx, y)(x - \bx) + \bigO\!\left(\Vert x - \bx\Vert^2\right)\!.
\end{aligned}
\end{align*}
By subtracting these two equalities 
and setting $y = \bx$ we obtain $R_1(x, y) - R_2(x, y) = \bigO\left(\Vert x - y\Vert^2\right)$.  
Furthermore,
\begin{align*}
&\lim_{\epsilon \rightarrow 0} ~ \frac{R_1(z, z\!+\!\epsilon d) - R_1(z, z)}{\epsilon}
\\
& \qquad =
\lim_{\epsilon \rightarrow 0} ~ \frac{R_2(z, z\!+\!\epsilon d) - R_2(z, z)}{\epsilon} + \bigO(\Vert \epsilon \Vert),
\end{align*}
which implies that $\dpartial{R_1}{y} (z, z) = \dpartial{R_2}{y} (z, z)$.
\end{proof}

Equipped with the above result, we now show that multiple shooting, single shooting and DDP locally converge at the same linear rate if a GGN Hessian approximation is used.
To this end, we summarize the primal variables as $w = (x, u)$ and introduce the following notation, where we set $N=2$ w.l.o.g. to keep the notation simple, 
{\setlength\arraycolsep{1.5pt}
\begin{align*}
\Flin(x, u; \bw) & = 
\begin{bmatrix}
\xzero & -& x_0\\
\bar{f}_0 + \bA_0 (x_0 - \bx_0) + \bB_0 (u_0 - \bu_0) & - & x_{1} \\
\bar{f}_1 + \bA_1 (x_1 - \bx_1) + \bB_1 (u_1 - \bu_1) & - & x_{2} \\
\end{bmatrix},
\\
F(x, u) &= 
\begin{bmatrix}
\xzero & -& x_0\\
f(x_0, u_0) & - & x_{1} \\
f(x_1, u_1) & - & x_{2} \\
\end{bmatrix},
\end{align*}
}%
denoting the linearized and nonlinear forward simulation, and
{\setlength\arraycolsep{1.5pt}
\begin{align*}
G(x, u; \bw) &= 
\begin{bmatrix}
\bu_0 + \bar k_0(\bw) + \bar K_0(\bw)(x_0 - \bx_0) &-& 	u_0 \\
\bu_1 + \bar k_1(\bw) + \bar K_1(\bw)(x_1 - \bx_1) &-& 	u_1
\end{bmatrix},
\end{align*}
}%
summarizing the feedback law.
The quantities $\bar k_i(\bw)$ and $\bar K_i(\bw)$ are computed according to \eqref{eq:k} and  \eqref{eq:K} respectively.
Note that they depend only on the primal iterate $\bw$ if a GGN Hessian is used.

\subsection{Multiple Shooting vs. DDP}

We define the \MS iteration as $w^+ = \Pi_\ms^\ggn(\bw)$ where $\Pi_\ms^\ggn(\bw)$ is the solution map of the linear root-finding problem $R_\ms^\ggn(w; \bw) = 0$ with
\begin{align} \label{eq:Rms1}
R_\ms^\ggn(w; \bw) 
= 
\begin{bmatrix}
\Flin(x, u; \bw) \\
G(x, u; \bw) 
\end{bmatrix}\!.
\end{align}

Similarly, we define the next DDP iterate via $w^+ = \Pi_\ddp^\ggn(\bw)$ where $\Pi_\ddp^\ggn(\bw)$ is the solution map of the nonlinear root-finding problem $R_\ddp^\ggn(w; \bw) = 0$ with
\begin{align} \label{eq:Rddp}
R_\ddp^\ggn(w; \bw) 
= 
\begin{bmatrix}
F(x, u) \\
G(x, u; \bw) \\
\end{bmatrix}\!.
\end{align}
Note that the only difference between \eqref{eq:Rms1} and \eqref{eq:Rss} is the first block where DDP uses the nonlinear dynamics while \MS uses the linearized dynamics.

\begin{proposition} \label{prop:ms-ddp}
Consider a KKT point $(w^*, \lambda^*)$ of the NLP in \eqref{eq:nlp} that satisfies LICQ and SOSC.
The asymptotic linear contraction rate $\kappa_\ms^\ggn(w^*)$ of the \MS iterates, obtained via $w^+ = \Pi_\ms^\ggn(w)$, is equal to the asymptotic linear contraction rate $\kappa_\ddp^\ggn(w^*)$ of the DDP iterates, obtained via $w^+ = \Pi_\ddp^\ggn(w)$.  
\end{proposition}
\begin{proof}
Theorem~\ref{thm:local-contraction} implies that it suffices to show that the partial derivatives of $R_\ms^\ggn(w; \bw)$ and $R_\ddp^\ggn(w; \bw)$ coincide at the fixed point $w^*$ in order to prove that $\kappa_\ms^\ggn(w^*) = \kappa_\ddp^\ggn(w^*)$.
We first consider the partial derivative w.r.t.~$\bw$.
From the definitions in \eqref{eq:Rms1} and \eqref{eq:Rddp} and together with
\begin{align*}
\dpartial{\Flin}{(x, u)}(x^*, u^*; w^*) = \dpartial{F}{(x, u)}(x^*, u^*; w^*),
\end{align*}
we directly obtain $\dpartial{R_\ms^\ggn}{w}(w^*; w^*) = \dpartial{R_\ddp^\ggn}{w}(w^*; w^*)$.
Together with Lemma~\ref{lem:partial-derivatives}, we conclude that also the partial derivatives w.r.t. $\bw$ coincide at $w^*$.
\end{proof}

\subsection{Multiple Shooting vs. Single Shooting}

Let $y=(x, \hat{x}, u)$.
We define the next \SS  iterate via $y^+ = \hat{\Pi}_\sss^\ggn(\by)$ where $\hat{\Pi}_\sss^\ggn$ is the solution map of the nonlinear root-finding problem $\hat{R}_\sss^\ggn(y; \by) = 0$ with
\begin{align} \label{eq:Rss}
\hat{R}_\sss^\ggn(y; \by) 
= 
\begin{bmatrix}
F(x, u) \\
\Flin(\hat{x}, u; \bw) \\
G^\ggn(\hat{x}, u; \bw) 
\end{bmatrix}\!.
\end{align}

Similarly, we define the next \MS iterate via $y^+ = \hat{\Pi}_\ms^\ggn(\by)$ where $\hat{\Pi}_\ms^\ggn$ is the solution map of the linear root-finding problem $\hat{R}_\ms^\ggn(y; \by) = 0$ with
\begin{align} \label{eq:Rms2}
\hat{R}_\ms^\ggn(y; \by) 
= 
\begin{bmatrix}
\Flin(x, u; \bw) \\
\Flin(\hat{x}, u; \bw) \\
G^\ggn(\hat{x}, u; \bw) 
\end{bmatrix}\!.
\end{align}
Note that the definition in \eqref{eq:Rms2} is redundant as it includes the same linear forward sweep twice. 
This is necessary only for the comparison with single shooting:
In this formulation, the two residual maps \eqref{eq:Rss} and \eqref{eq:Rms2} differ only in the first block where single shooting uses a nonlinear forward simulation while multiple shooting performs the forward simulation based on the linearized dynamics.

\begin{proposition} \label{prop:ms-ss}
Consider a KKT point $(w^*\!, \lambda^*)$, $w^* \!=\! (x^*\!, u^*)$ of the NLP  \eqref{eq:nlp} that satisfies LICQ and SOSC and let $y^* = (x^*, x^*, u^*)$.
The asymptotic linear contraction rate $\kappa_\ms^\ggn(y^*)$ of the \MS iterates, obtained via $y^+ = \hat{\Pi}_\ms^\ggn(y)$, is equal to the asymptotic linear contraction rate $\kappa_\sss^\ggn(y^*)$ of the \SS  iterates, obtained via $y^+ = \Pi_\sss^\ggn(y)$.  
\end{proposition}
\begin{proof}
We proceed as in the proof of Proposition~\ref{prop:ms-ddp} and show that the partial derivatives of of $\hat{R}_\ms^\ggn(y; \by)$ and $\hat{R}_\sss^\ggn(y; \by)$ coincide at the fixed point $y^*$.
From the definitions in \eqref{eq:Rms2} and \eqref{eq:Rss} and together with
\begin{align*}
\dpartial{\Flin}{(x, u)}(x^*, u^*; z^*) = \dpartial{F}{(x, u)}(x^*, u^*; z^*),
\end{align*}
we directly obtain $\dpartial{\hat{R}_\ms^\ggn}{y}(y^*; y^*) = \dpartial{\hat{R}_\sss^\ggn}{y}(y^*; y^*)$.
Together with Lemma~\ref{lem:partial-derivatives}, we conclude that the partial derivatives with respect to $\by$ coincide at $y^*$,  $\dpartial{\hat{R}_\ms^\ggn}{\by}(y^*; y^*) = \dpartial{R_\sss^\ggn}{\by}(y^*; y^*)$.
\end{proof}

\subsection{Local Linear Contraction Rate}

In the previous section, we have shown that multiple shooting, single shooting and DDP locally converge with the same linear rate if a GGN Hessian is used.
We now analyze the \MS iteration to further characterize this linear rate.
To this end, we consider yet another equivalent definition of the \MS iteration.
We define the next \MS iterate via $z^+ = \tilde{\Pi}_\ms^\ggn(\bz)$ where $\tilde{\Pi}_\ms^\ggn$ is the solution map of the linear root-finding problem $\tilde{R}_\ms^\ggn(z; \bz) = 0$ with
\begin{align} \label{eq:Rms-three}
\tilde{R}_\ms^\ggn(z; \bz) 
= 
\begin{bmatrix}
	V_{\qua}^\ggn(x, u; \bw) + \grad{w} \Flin(x, u; \bw) \lambda \\
	\Flin(x, u; \bw),
\end{bmatrix}
\end{align}
where $V_{\qua}^\ggn(x, u; \bw)$ is the quadratic approximation of the objective given in \eqref{eq:qp} using a GGN Hessian, i.e.,
\begin{align}
\begin{aligned}
V_{\qua}^\ggn(w; \bw) &= V(\bw) + \tfrac{\partial V}{\partial w}(\bw)(w-\bw) \\
&\qquad + \tfrac{1}{2}\left(w-\bw\right)\T H_\ggn(\bw) (w - \bw).	
\end{aligned}
\end{align}

\begin{theorem}
Consider a KKT point $z^* = (x^*, u^*, \lambda^*)$ of the NLP in \eqref{eq:nlp} that satisfies LICQ and SOSC and let $y^* = (x^*, x^*, u^*)$.
The asymptotic linear contraction rate of the \MS iterates, the \SS  iterates, and the DDP iterates is the same for all three methods and given by the smallest $\kappa$ that satisfies the condition
\begin{align} \label{eq:kappaLMI}
-\kappa \tilde{M}^\ggn(z^*) \preceq \tilde{E}^\ggn(z^*) \preceq \kappa \tilde{M}^\ggn(z^*),
\end{align}
with $\tilde{M}^\ggn(z^*)\! =\! Z\T M^\ggn(z^*) Z$, $\tilde{E}^\ggn(z^*) \!=\! Z\T E^\ggn(z^*) Z$ with $Z$ a basis of the null space of $\grad{w} F(x^*, u^*)\T$. 
Here $M^\ggn(z^*)$ is the GGN Hessian approximation and $E^\ggn(z^*)$ is the deviation from the exact Hessian. 

\end{theorem}

\begin{proof}	
The fact that the local linear contraction rate is the same for all three methods follows from Proposition \ref{prop:ms-ddp} and \ref{prop:ms-ss}. 
The characterization of the rate in terms of the linear matrix inequalities in \eqref{eq:kappaLMI} follows from Theorem 4.5 in \cite{Messerer2021a} applied to the root-finding problem \eqref{eq:Rms-three}.
\end{proof}

\begin{remark}
Note that for the OCP-structured problem at hand, the error matrix $E^\ggn(z^*)$ is given as
\begin{align} \label{eq:error}
E^\ggn(z^*) = \grad{(x, u)}^2 \left((\lambda^*)\T F(x^*, u^*) \right).
\end{align}
Considering $N=3$ and reordering the optimization variables as $\tilde{w}=(x_0, u_0, x_1, u_1, x_2, u_2, x_3)$, a basis $Z$ of the null space of $\grad{\tilde{w}} F(x^*, u^*)\T$ is given by
\begin{align*}
\fontsize{9.9}{11}\selectfont
Z = 
\begin{bmatrix}
\zeros & \zeros & \zeros\\
\I & \zeros & \zeros\\
B_0^* & \zeros & \zeros\\
\zeros & \I  & \zeros\\
A_1^*B_0^* & B_1^* & \zeros\\
 \zeros &\zeros & \I  \\
A_2^*A_1^*B_0^* & A_2^*B_1^* & B_2^*\\
\end{bmatrix}.
\end{align*}
Here, the Jacobians are given as $A_i^* = \tfrac{\partial f}{\partial x}(x_i^*, u_i^*)$ and $B_i^* = \tfrac{\partial f}{\partial u}(x_i^*, u_i^*)$.
The reduced error matrix $\tilde{E}^\ggn(z^*) = Z\T E^\ggn(z^*) Z$ corresponds to the error matrix of the single shooting problem if solved in the standard dense formulation.
\end{remark}

\begin{corollary}
Propositions~\ref{prop:ms-ddp} and \ref{prop:ms-ss} imply that
\begin{align}
\Vert w^+_{\ms} - w^+_{\sss} \Vert_2 = \bigO\left(\Vert \bar w - w^*\Vert_2^2\right)\!,
\end{align}
if both methods start from a feasible initial guess $\bar w$.
The analogous result holds for multiple shooting vs. DDP and single shooting vs. DDP.
\end{corollary}

\subsection{Exact Hessian Variants and Quadratic Rate}

If an exact Hessian is used, multiple shooting, single shooting and DDP  locally converge at a quadratic rate.
In \cite{Albersmeyer2010}, the local convergence behaviour of \MS and \SS with exact Hessians has been discussed in a simplified setting.
The authors concluded that \MS formulations lead to faster contraction rates if the nonlinearities of the system dynamics reinforce each other, i.e. have the same direction of curvature.
Single shooting would lead to faster contraction if the concatenated nonlinearities mitigate each other.
They furthermore argued that in optimal control, the nonlinearities often reinforce each other rendering a \MS approach favorable.

\section{Further Remarks \& Discussion}

Our main result shows that if a GGN Hessian is used all three methods locally behave the same.
In the following, we discuss further differences and similarities, as well as advantages and disadvantages of the three methods.

\subsection{Initialization}

\CMS can start from an infeasible initial guess, which is not possible for single shooting and DDP.
The possibility for infeasible initialization greatly simplifies the usage of \MS in practice as no additional routines for finding a feasible initial guess are required.
Furthermore, infeasible initialization might improve convergence of the method if a rough guess of how a solution trajectory might look like is available \cite{Osborne1969}. 


\subsection{GGN Hessian \& Convex-Over-Nonlinear Objectives}

We focused on a GGN Hessian approximation which comes with several advantages:
(1) The subproblems that need to be solved in each iteration are convex and can thus be solved reliably.
(2) We require only first-order derivatives of the dynamics. As second-order derivatives tend to be expensive to compute, using a GGN Hessian might significantly reduce computational cost of the method.
(3) The GGN Hessian does not require computation of the multipliers $\lambda$ reducing the complexity of the algorithm.

Moreover, the GGN approach naturally covers convex-over-nonlinear cost where $l_i$ is not convex but instead takes the form $l_i(x_i, u_i) = \psi_i(r_i(x_i, u_i))$ with $\psi_i$ convex and $r_i$ nonlinear. 
In this case, the GGN Hessian approximation is
\begin{align}
\begin{bmatrix}
Q_i^\ggn & (S_i^\ggn)\T \\ 
S_i^\ggn & R_i^\ggn
\end{bmatrix}
= \bar{J}_i(\bx_i, \bu_i)\T \grad{}^2 \phi_i(\bar{r}_i) \bar{J}_i(\bx_i, \bu_i),
\end{align}
where $\bar{r}_i = r_i(\bx_i, \bu_i)$ and $\bar{J}_i(\bx_i, \bu_i) =  \dpartial{r_i}{(x_i, u_i)}$.
If the Hessian error matrix in \eqref{eq:error} is adapted accordingly, our local convergence analysis is still valid in this more general case.

\subsection{Constraints}

Constraints might be incorporated into the problem formulation via barrier functions or penalty functions as is done e.g. in \cite{Wills2004, Zanelli2017a, Feller2017a, Marti-Saumell2020, Baumgaertner2022a}.
Note that these formulations typically lead to convex-over-nonlinear objective functions.

\subsection{Implementation Details}

\subsubsection{Forward Sweep}

In the form presented here, all three methods can be implemented in a very similar fashion if a GGN Hessian is used.
With the Riccati recursion being the same for the three methods, only the forward sweep needs to be adapted.
One advantage of \MS is the possibility to parallelize the forward sweep. 

\subsubsection{Line Search}

If a line search strategy is used \eqref{eq:u-forward} needs to be changed to
\begin{align*}
\text{DDP \& multiple shooting}&&	u_i & = \bu_i + \alpha k_i + K_i(x_i - \bx_i), \\
\text{single shooting:}	&& u_i & = \bu_i + \alpha k_i + K_i(\hat{x}_i - \bx_i), 
\end{align*}
where $\alpha \in (0,1]$ is the step size that is successively reduced until a sufficient decrease criterion is met.
For single shooting and DDP, the iterates are always feasible such that simply the cost function can be considered within a sufficient decrease approach.
For multiple shooting, however, one needs to decide on a merit function in order to combine both sufficient decrease of the cost function and infeasibility reduction into a scalar progress criterion.

%
%

\section{Numerical Results}
\label{sec:examples}

In this section, we illustrate the local convergence rate of the multiple shooting, single shooting, and DDP iterates on a numerical example adopted from \cite{Chen1998}.
The continuous time dynamics are given as
\begin{align*}
\dot{x}_1  = x_2 \!+\! u\left(\mu\!+\!(1\!-\!\mu)x_1\right)\!, ~~
\dot{x}_2  = x_1 \!+\! u\left(\mu\!-\!4(1\!-\!\mu)x_2\right)\!,
\end{align*}
with $\mu = 0.7$.
We discretize the continuous dynamics using ten steps of a Runge-Kutta integrator of fourth order on an integration interval of $h=0.25$.
We use a quadratic cost and a quadratic penalty which penalizes inputs that do not satsify $|u_i| \leq u_{\max} = 1$, yielding the following objective:
\begin{align*}
V(x, u) =  \sum_{i=0} \tfrac{1}{2}x_i\T Q x_i + \tfrac{1}{2}u_i\T R u_i + \tau \beta(u_i)  + \tfrac{1}{2} x_N\T P x_N   
\end{align*} 
with $Q = \text{diag}(0.5, 0.5)$, $R=0.8$, $P=\text{diag}(10, 10)$ and penalty $\beta(u_i) = \max(0, u_i-u_{\max})^2 + \min(0, u_i+u_{\max})^2$ with $\tau=100$.
Note that the objective $V(x, u)$ comprising the cost and penalty functions is convex and a GGN Hessian can be used.
In the following, we obtain a feasible initial guess by simulating the nonlinear system forward in time using the linear feedback law obtain from an LQR which is based on the linear system obtained by linearizing at the steady state $x_{\mathrm{steady}} = 0$, $u_{\mathrm{steady}} = 0$.

In Fig.~\ref{fig:x-traj}, the state trajectory after a single iteration of multiple shooting, single shooting, and DDP is shown for a horizon length of $N=20$. 
All three methods start from the same feasible initial guess that is shown in gray.
While both the \SS  and DDP trajectory are feasible with respect to the dynamics constraints, the \MS trajectory exhibits gaps.
The difference of the \SS  iterate to the \MS iterate increases significantly along the horizon.

In Fig.~\ref{fig:emp-contraction}, the empirical linear contraction rate, which is defined as
$
\hat{\kappa} = \frac{\Vert w^{k+1} - w^k\Vert}{\Vert w^k - w^{k-1}\Vert},
$
is shown.
The exact Hessian variants require only very few iterations to reach the convergence criterion and the empirical contraction rate quickly tends to zero.
The empirical contraction rate of the three GGN variants needs a few more iterations to converge and they converges to the same value. 

Fig.~\ref{fig:quad-conv-illustration} shows the distance to the solution of the iterate as a function of the distance to the solution of the previous iterate.
In this log-log plot, the slope of the linear functions corresponds to the order of convergence.
The intercepts correspond to the rate of convergence.
The GGN variants of the three methods converge linearly at exactly the same asymptotic rate, while the three methods using an exact Hessian converge quadratically.

\begin{figure}[t]
	\centering
	\includegraphics[width=0.9\linewidth]{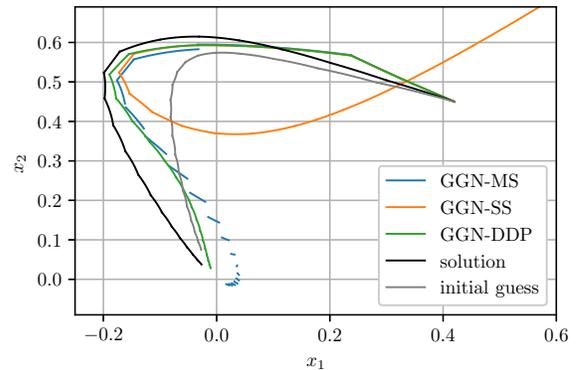}
	\vspace*{-5pt}
	\caption{
		State trajectories obtained after one iteration of GGN-MS, GGN-SS, and GGN-DDP.
		All methods start from the same feasible initial guess. 
		The initial state is $\xzero = (0.42, 0.45)$.
	}
	\label{fig:x-traj}
\end{figure}

\begin{figure}[t]
	\centering
	\includegraphics[width=0.9\linewidth]{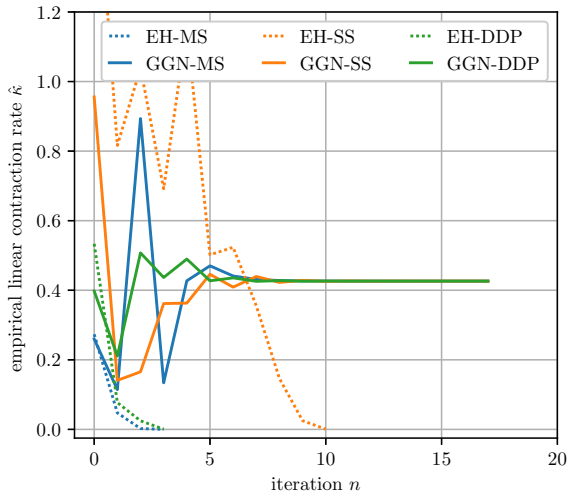}
	\vspace*{-5pt}
	\caption{
		Empirical contraction rate $\hat{\kappa}$ for the exact Hessian and GGN Hessian variant of MS, SS and DDP.
		The convergence criterion is $\Delta w \leq  10^{\sm 12}$.
	}
	\label{fig:emp-contraction}
\end{figure}

\begin{figure}[t]
	\centering
	\includegraphics[width=0.85\linewidth]{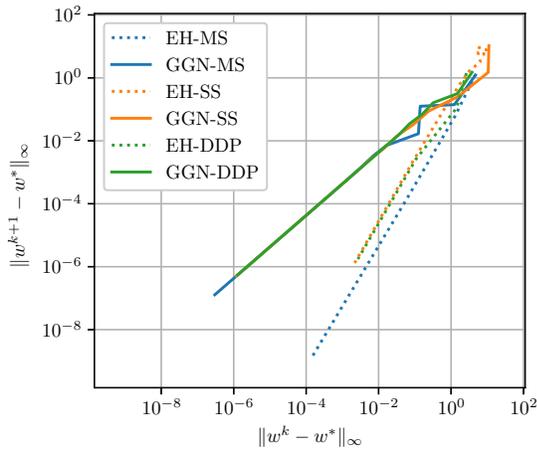}
	\vspace*{-5pt}
	\caption{
		Norm of the primal step as a function of the norm of the previous primal step.
		The slope of the linear function corresponds to the order of convergence.
		The intercept corresponds to the rate of convergence.
	}
	\label{fig:quad-conv-illustration}
\end{figure}

\section{Conclusions and Outlook}

We provided a unified view on multiple shooting, single shooting and DDP -- three classical methods for solving unconstrained discrete optimal control problems -- from a numerical optimization perspective.
Depending on the choice of Hessian approximation used within these methods, different algorithmic variants may be obtained.
We focused on the Generalized Gauss-Newton (GGN) Hessian, a generalization of the widely used Gauss-Newton (GN) Hessian, to convex loss functions.
We showed that the iterates obtained with the three methods locally converge – or diverge – to a KKT point of the OCP a the same linear rate.

\bibliography{extracted}

\end{document}